\documentclass{amsart}
\usepackage{mathrsfs}

\usepackage{amssymb,latexsym}     
\usepackage[utf8]{inputenc}
\usepackage[english]{babel}

\usepackage{mathrsfs}
\usepackage{amsmath,amsthm}
\usepackage{graphicx}
\usepackage{tikz-cd}
\usepackage{adjustbox}
\usepackage{tikz}
\usetikzlibrary{matrix}
\usepackage{caption}

\newtheorem{Lemma}{Lemma}[section]

\newtheorem{Theorem}[Lemma]{Theorem}

\newtheorem{Corollary}[Lemma]{Corollary}
\theoremstyle{definition}
\newtheorem{Remark}[Lemma]{Remark}
\newtheorem{Definition}[Lemma]{Definition}

\newtheorem{innercustomthm}{Theorem}
\newenvironment{Theoremn}[1]
  {\renewcommand\theinnercustomthm{#1}\innercustomthm}
  {\endinnercustomthm}

\usepackage[lmargin=1in,rmargin=1in,bmargin=0.5in,tmargin=1in]{geometry}
\begin{document}

\title{Generalized	Jacobi identities and Jacobi elements \\ of the group ring of the symmetric group }

\author{Sergei O. Ivanov}
\address{Chebyshev Laboratory, St. Petersburg State University, 14th Line, 29b,
Saint Petersburg, 199178 Russia} 
\email{ivanov.s.o.1986@gmail.com}

\author{Savelii Novikov}
\address{Saint Petersburg, Russia}
\email{novikov.savelii00@gmail.com}

\begin{abstract}
By definition the identities $[x_1, x_2] + [x_2, x_1] = 0$ and
$[x_1, x_2, x_3] + [x_2, x_3, x_1] + [x_3, x_1, x_2] = 0$ hold in any Lie algebra. It
is easy to check that the identity $[x_1, x_2, x_3, x_4] + [x_2, x_1, x_4, x_3] +
[x_3, x_4, x_1, x_2] + [x_4, x_3, x_2, x_1] = 0$ holds in any Lie algebra as well.
I. Alekseev in his recent work introduced the notion of Jacobi subset of the symmetric group $S_n$. It is a subset of $S_n$ that gives an identity of this kind. We introduce a notion of Jacobi element of the group ring $\mathbb{Z}[S_n]$ and describe them on the language of equations on coefficients. Using this description we obtain a purely combinatorial necessary and sufficient condition for a subset to be Jacobi. 
\end{abstract}
\maketitle
\section*{Introduction}
By a Lie ring we mean a Lie algebra over $\mathbb{Z}$. Any Lie algebra can be considered as a Lie ring.
By definition the identities $[x_1, x_2] + [x_2, x_1] = 0$ and 
$[x_1, x_2, x_3] + [x_2, x_3, x_1] + [x_3, x_1, x_2] = 0$ hold in any Lie ring where $[x_1, . . . , x_n]$ 
denotes the left-normed bracket: $[x_1, . . . , x_n]=[[x_1, . . . , x_{n-1}],x_n]$. Moreover, it is easy to check that there is one more identity that 
holds in any Lie ring: $[x_1, x_2, x_3, x_4]+[x_2, x_1, x_4, x_3] + [x_4, x_3, x_2, x_1] + [x_3, x_4, x_1, x_2] = 0$. 
Following I. Alekseev we call a subset $T\subseteq S_n$  \textit{Jacobi} if the following identity 
holds in any Lie ring: $$\sum\limits_{\sigma\in S_n} [x_{\sigma(1)}, \dots , x_{\sigma(n)}] = 0.$$
In this paper we study a more general notion.
An element of the group ring $a \in \mathbb{Z}[S_n]$ is called  \textit{Jacobi element} if the following identity is satisfied in any Lie ring. 
\[
\sum\limits_{\sigma\in S_n} a(\sigma) [x_{\sigma(1)}, \dots , x_{\sigma(n)}] = 0
\] 
By $\mathsf{J}_n$ we denote the set formed by all Jacobi elements of $\mathbb{Z}[S_n]$.

Throughout the paper we use the following definition of $(s,t)$--shuffle. It is a pair $(\alpha,\beta)$ such that $\alpha:\{1,\dots,s\} \to \{1,\dots, s+t\}$ and
$\beta:\{1,\dots,t\}\to\{1,\dots, s+t\}$ are strictly monotonic functions with non intersecting images. 
The set of all $(s,t)$--shuffles is denoted by ${\sf Sh}(s,t)$. By ${\sf Sh}^1(s,t)$ we denote the set  of all $(s,t)$--shuffles with $\alpha(1)=1$. 
One of the main results is the following description of all Jacobi elements of $\mathbb{Z}[S_n]$.

\begin{Theoremn}{1}{\it Let $\lambda= \sum\limits_{\sigma\in S_n} \lambda(\sigma) \sigma $ be an element of $ \mathbb Z[S_n].$ 
Then $\lambda$ is a Jacobi element if and only if for any $\tau\in S_n$
$$  \sum\limits_{i=0}^{n-1} 
\sum\limits_{(\alpha,\beta)} (-1)^i \lambda\left(\tau\cdot\left(
\begin{smallmatrix} 
1 & \dots & i & i+1 & \dots & n \\
\beta(i) & \dots & \beta(1) & \alpha(1) & \dots & \alpha(n-i)
\end{smallmatrix}\right)^{-1}\right)=0,$$
where $(\alpha,\beta)$ runs over  ${\sf Sh}^{1}(n-i,i).$  }
\end{Theoremn}

The main result of this paper is a necessary and sufficient condition for a subset to be Jacobi.
Consider two subsets of $S_n$: 
$$\mathrm{I}_n^+ =\left\{\left(
\begin{smallmatrix} 
1 & \dots & i & i+1 & \dots & n \\
\beta(i) & \dots & \beta(1) & \alpha(1) & \dots & \alpha(n-i)
\end{smallmatrix}\right)^{-1}\mid (\alpha,\beta) \in Sh^1(n-i, i), \,i \mbox{ is even}, \ 0\leq i\leq n-1 \right\},$$
$$\mathrm{I}_n^- =\left\{\left(
\begin{smallmatrix} 
1 & \dots & i & i+1 & \dots & n \\
\beta(i) & \dots & \beta(1) & \alpha(1) & \dots & \alpha(n-i)
\end{smallmatrix}\right)^{-1}\mid (\alpha,\beta) \in Sh^1(n-i, i), \,  i \mbox{ is odd}, \ 0\leq i\leq n-1 \right\}.$$

\begin{Theoremn}{2}{\it 
Let $T$ be a subset of $S_n$.  Then $T$ is Jacobi if and only if for any $\tau\in S_n$
$$  \mid T \cap \tau \mathrm{I}_n^+ \mid = 
\mid T \cap \tau \mathrm{I}_n^- \mid.$$}
\end{Theoremn}
\noindent Using this theorem we obtain the following interesting identities that hold in any Lie ring.

$$[x_1,x_2,x_3,x_4]+[x_3,x_4,x_1,x_2]+[x_2,x_1,x_4,x_3]+[x_4,x_3,x_2,x_1]=0$$ 
$$[x_1,x_2,x_3,x_4]+[x_3,x_1,x_2,x_4]+[x_4,x_1,x_2,x_3]+[x_1,x_4,x_3,x_2]+[x_2,x_3,x_4,x_1]=0$$
$$[x_1,x_2,x_3,x_4]+[x_3,x_1,x_2,x_4]+[x_2,x_1,x_4,x_3]+[x_4,x_2,x_1,x_3]+[x_1,x_4,x_3,x_2]
+[x_2,x_3,x_4,x_1]=0$$
{\small $$ [x_1,x_2,x_3,x_4]+[x_3,x_1,x_2,x_4]+[x_2,x_1,x_4,x_3]+[x_4,x_2,x_1,x_3]+[x_1,x_3,x_4,x_2]+[x_3,x_4,x_1,x_2]+[x_2,x_3,x_4,x_1]=0$$}

\noindent The paper is organised as follows. In the first section we describe Jacobi elements and prove Theorem 1 with some additional statements. In the second section we consider Jacobi subsets using the language of Jacobi elements and prove Theorem 2. 

\section{Jacobi elements}

\begin{Definition} An element $a=\sum_{\sigma\in S_n} a(\sigma) \sigma$ of $\mathbb{Z}[S_n]$ is called a \textit{Jacobi element} if the following identity is satisfied in any Lie ring. 
$$\sum\limits_{\sigma\in S_n} a(\sigma) [x_{\sigma(1)},\dots , x_{\sigma(n)}] = 0$$ 
The set of all Jacobi elements of $\mathbb Z[S_n]$ is denoted by $\mathsf{J}_n$.
\end{Definition}

\begin{Definition} Let $G$ be a finite group, $\mathbb Z[G]$ be the group ring and $a=\sum_{g\in G} a(g)g, b=\sum_{g\in G} b(g)g$ be elements of $\mathbb Z[G].$  The scalar product of $a$ and $b$ is defined as follows
$$\langle a,b \rangle=\sum_{g\in G} a(g)\cdot b(g).$$
(By the scalar product in the group ring some authors mean  $\sum_{g\in G} a(g)\cdot b(g^{-1})$. We do not use this definition.) 
\end{Definition}
\noindent  The {\it antipode} is the map 
$$\mathsf{s}:\mathbb{Z}[G]\rightarrow\mathbb{Z}[G],\mbox{ such that } \mathsf{s}(a)=\sum\limits_{g\in G}a(g)g^{-1},$$
where $a=\sum\limits_{g\in G}a(g)g.$

\begin{Lemma}\label{tri} Let $a,b,c\in \mathbb Z[G].$ Then 
$$ \langle a\cdot b,c \rangle = \langle a,c\cdot{\sf s}(b)  \rangle.$$
\end{Lemma}
\begin{proof}
Computations show that 
$$\langle a\cdot b,c \rangle= \sum_{g\in G} (a\cdot b)(g)\cdot c(g)= \sum_{g\in G}\ \sum_{\substack{f,h\in G \\ f \cdot h = g}} a(f)\cdot b(h)\cdot c(g)$$ 
$$\langle a,c\cdot {\sf s}(b)  \rangle = \sum_{g\in G} a(g)\cdot (c\cdot {\sf s}(b))(g)=\sum_{g\in G}\ \  \sum_{\substack{f,h\in G \\ f \cdot h = g}} a(g)\cdot b(h^{-1})\cdot c(f).$$
Then both scalar products are the same sums of elements $a(g_1)\cdot b(g_2)\cdot c(g_3)$, where $g_1, g_2, g_3 \in G$ and $g_1 \cdot g_2 = g_3$. 
\end{proof}

\noindent If $X\subseteq \mathbb{Z}[G]$ we set $X^\perp=\{a\in \mathbb Z[G]\mid \forall x\in X\ \langle a,x \rangle=0 \}$ and call it  {\it orthogonal complement}. 

\begin{Remark}\label{rem}
For any $X\subseteq\mathbb{Z}[G]$ the following is satisfied:
$$X^\perp=\; \langle X \rangle^\perp,$$ 
where $\langle X \rangle$ is the abelian subgroup generated by $X$.
\end{Remark}

\begin{Lemma}\label{mlemma}
Let $G$ be a finite group and $\lambda\in \mathbb Z[G]$. Then the following is satisfied:
$$Ker\, (\cdot\,\lambda) = \left(G\cdot{\sf s}(\lambda)\right)^{\perp},$$
where $\cdot\,\lambda$ is a map such that $(\cdot\,\lambda)(a)=a\cdot\lambda$ for every $a\in\mathbb{Z}[G]$.
\end{Lemma}

\begin{proof}
First, we need to prove that 
$Ker\, (\cdot\,\lambda) = \left(\mathbb{Z}[G]\cdot{\sf s}(\lambda)\right)^{\perp}$. Consider $a\in Ker\, (\cdot\,\lambda)$ then $a\cdot\lambda=0$. According to basic properties of scalar product and lemma \ref{tri} the following is satisfied:
\begin{align*}
   \forall x\in \mathbb{Z}[G]\quad&\langle x,a\cdot\lambda\rangle=0 \\
   								  &\langle a\cdot\lambda,x\rangle=0 \\
   								  &\langle a,x\cdot\mathsf{s}(\lambda)\rangle=0
\end{align*}
Hence $a\in \left(\mathbb{Z}[G]\cdot{\sf s}(\lambda)\right)^{\perp}$. Consider 
$b\in \left(\mathbb{Z}[G]\cdot{\sf s}(\lambda)\right)^{\perp}$ then $\forall c\in \mathbb{Z}[G]\cdot
\mathsf{s}(\lambda) \:\: \exists x \in \mathbb{Z}[G]\:\: c=x\cdot\mathsf{s}(\lambda)$ 
\begin{align*}
   							      & \langle b,x\cdot\mathsf{s}(\lambda)\rangle=0 \\
   								  &\langle b\cdot\lambda,x\rangle=0 
\end{align*}
Hence $b\cdot\lambda=0$ and $b\in Ker\, (\cdot\,\lambda)$. It is obvious that $ \mathbb{Z}[G]\cdot{\sf s}(\lambda)=\langle G\cdot{\sf s}(\lambda)\rangle$ then according to remark \ref{rem} and statement above
$$Ker\, (\cdot\,\lambda)=\left(\mathbb{Z}[G]\cdot{\sf s}(\lambda)\right)^{\perp}=\left(G\cdot{\sf s}(\lambda)\right)^{\perp}.$$  
\end{proof}

\noindent By $\omega_n$ we denote an element of $\mathbb{Z}[S_n]$ which is defined as follows: 
$$\omega_n = \sum_{i=0}^{n-1}\sum_{(\alpha,\beta)}(-1)^i
\left(
\begin{array}{cccccc}
1 & \dots & i & i+1 & \dots & n \\
\beta(i) & \dots & \beta(1) & \alpha(1) & \dots & \alpha(n-i)
\end{array}
\right),$$ 
where $(\alpha,\beta)\in Sh^1(n-i,i)$. 
 
\noindent In the work we used the following notation. 

\noindent $\gamma_n$ is the abelian subgroup of the free associative ring  $
\mathbb Z\langle x_1,\dots,x_n\rangle$ generated by monomials $ x_{\sigma(1)}\dots x_{\sigma(n)},$ where $ \sigma\in S_n.$ 

\noindent $\beta_n:\gamma_n\longrightarrow\gamma_n$ is the  homomorphism  that is defined on the basis as  
$\beta_n(x_{\sigma(1)}\dots x_{\sigma(n)}) = [x_{\sigma(1)},\dots ,x_{\sigma(n)}].$

\noindent  $\widetilde\beta_n:\mathbb{Z}[S_n]\longrightarrow\gamma_n$ is  the  homomorphism  such that 
$\widetilde{\beta_n}(\sum_{\sigma\in S_n}\alpha(\sigma)\sigma)=\sum_{\sigma\in S_n}\alpha(\sigma) [x_{\sigma(1)}, \dots  , x_{\sigma(n)}].$ 
	
\noindent $\varphi:\gamma_n\longrightarrow\mathbb{Z}[S_n]$ is the isomorphism given by: 
$$\varphi(\sum_{\sigma\in S_n}\alpha(\sigma)\cdot x_{\sigma(1)}\dots x_{\sigma(n)})=
\sum_{\sigma\in S_n}\alpha(\sigma)\sigma.$$ 
 
\noindent It is important to mention that $\varphi([x_1,\dots,x_n])=\omega_n$. 

\noindent $\Omega_{n}:\mathbb{Z}[S_n]\rightarrow\mathbb{Z}[S_n]$ is a homomorphism such that for all 
$a\in \mathbb{Z}[S_n] \quad \Omega_{n}(a)=a\cdot\omega_n$. 

\noindent There is a connection between all these maps.

\begin{Lemma}\label{diag} The following diagram is commutative: 
\begin{center}
\adjustbox{scale=1.3,center}{
\begin{tikzcd}
    \gamma_n \arrow{r}{\beta_n} \arrow[swap]{d}{\varphi} & \gamma_n \arrow{d}{\varphi}\\  
    \mathbb{Z}[S_n]\arrow[dotted]{ur}{\widetilde\beta_n} \arrow{r}[swap]{\Omega_n} & \mathbb{Z}[S_n].
\end{tikzcd}}
\end{center}
\end{Lemma}

\begin{proof}
To prove the fact that this diagram is commutative we need to show that $\widetilde\beta_n\circ\varphi=\beta_n$ and $\varphi\circ\widetilde\beta_n=\Omega_n$. \\ 
\textit{The first equality.}  
Consider an element $g\in\gamma_n$ then $g=\sum\limits_{\sigma\in S_n}g(\sigma) x_{\sigma(1)} {\dots} x_{\sigma(n)}$. It is obvious that
\begin{align*}
\varphi(g) =\sum\limits_{\sigma\in S_n}g(\sigma)\sigma \\
\widetilde\beta_n(\varphi(a)) =\sum\limits_{\sigma\in S_n}g(\sigma) [x_{\sigma(1)},\dots,&x_{\sigma(n)}]=\beta_n(g)
\end{align*} \noindent \textit{The second equality.}
Consider an element $a\in\mathbb{Z}[S_n]$ then $a=\sum_{\sigma\in S_n}a(\sigma)\sigma$. 
We can assume that $\widetilde\beta_n(a)=\sum_{\sigma\in S_n}a(\sigma) [x_{\sigma(1)},
\dots,x_{\sigma(n)}]$ and $\varphi([x_{\sigma(1)},\dots,x_{\sigma(n)}])=\sigma\cdot\omega_n$ (see reference \ref{Ilya}, identity (2.1)). It is correct because $\sigma$ permutes corresponding letters (indices) in obtained formula for bracket $[x_1, \dots, x_n]$. Hence the following is satisfied: 
\begin{align*}
\varphi(\widetilde\beta_n(a))=\sum\limits_{\sigma\in S_n}a(\sigma)(\sigma\cdot\omega_n)=\Omega_n(a).
\end{align*}
\end{proof}

\begin{Lemma}\label{ker}
The following is satisfied:
$$\mathsf{J}_n=Ker\;\tilde\beta_n=Ker\;\Omega_n$$
\end{Lemma}

\begin{proof}  ${\sf J}_n$ consists of elements $\sum a(\sigma) \sigma \in \mathbb Z[S_n]$ such that $\sum a(\sigma)[l_{\sigma(1)},\dots, l_{\sigma(n)}]=0$ for any elements $l_1,\dots, l_n$ of any Lie ring $L.$ Denote by $\mathcal L(x_1,\dots,x_n)$ the Lie subalgebra of $\mathbb Z\langle x_1,\dots,x_n \rangle$ generated by $x_1,\dots,x_n.$ It is the free Lie ring generated by $x_1,\dots,x_n$ (see \cite{Reutenauer}). Then there is a homomorphism of Lie rings $\Phi: \mathcal L(x_1,\dots,x_n)\to L$ such that $\Phi(x_i)=l_i.$ It follows that an element $a$ of $\mathbb Z[S_n]$ lies in ${\sf J}_n$ if and only if $\sum a(\sigma)[x_{\sigma(1)},\dots, x_{\sigma(n)}]=0$ in $\mathcal L(x_1,\dots,x_n).$ Then ${\sf J}_n=Ker \:\tilde \beta_n.$

To prove the second equality we need to use Lemma \ref{diag}. According to it $\Omega_n=\varphi\circ\widetilde\beta_n$ and $\mathsf{J}_n=Ker\,\widetilde\beta_n$.
Since $\varphi$ is an isomorphism then $\varphi(0)=0$ hence for all  
$a\in \mathsf{J}_n=Ker\,\widetilde\beta_n$ 
$$\Omega_n(a)=\varphi(\widetilde\beta_n(a))=\varphi(0)=0\;\Longrightarrow \; a\in Ker\,\Omega_n$$
Consider an element $b\in Ker\;\Omega_n$ then $\Omega_n(b)=\varphi(\widetilde\beta_n(b))=0$. 
Hence $\widetilde\beta_n(b)=0$ and $b\in\mathsf{J}_n$.  
\end{proof}

\begin{Corollary}\label{col}
 $$\mathsf{J}_n=Ker\;\tilde\beta_n=Ker\;\Omega_n=\left(S_n\cdot s(\omega_n)\right)^\perp.$$
\end{Corollary}

\begin{proof}
$S_n$ is a finite group and $\omega_n\in \mathbb{Z}[S_n]$ then conditions of Lemma \ref{mlemma} are fulfilled 
and according to lemma \ref{ker} we get the assertion. 
\end{proof}

\begin{Theorem}\label{th1} Let $\lambda= \sum\limits_{\sigma\in S_n} \lambda(\sigma) \sigma $ be an element of $ \mathbb Z[S_n].$ 
Then $\lambda$ is a Jacobi element if and only if for any $\tau\in S_n$
$$  \sum\limits_{i=0}^{n-1} 
\sum\limits_{(\alpha,\beta)} (-1)^i \lambda\left(\tau\cdot\left(
\begin{smallmatrix} 
1 & \dots & i & i+1 & \dots & n \\
\beta(i) & \dots & \beta(1) & \alpha(1) & \dots & \alpha(n-i)
\end{smallmatrix}\right)^{-1}\right)=0,$$
where $(\alpha,\beta)$ runs over  ${\sf Sh}^{1}(n-i,i).$    
\end{Theorem}

\begin{proof}
Consider $\lambda= \sum\limits_{\sigma\in S_n} \lambda(\sigma) \sigma\in \mathbb{Z}[S_n]$. Then 
$\lambda\in {\sf J}_n$ is equivalent to $\lambda\in(S_n\cdot\mathsf{s}(\omega_n))^\perp$ according to corollary \ref{col}. Hence for all $a=\tau\cdot\mathsf{s}(\omega_n)\in (S_n\cdot\mathsf{s}(\omega_n))$ it is true that
$\langle \lambda,\tau\cdot\mathsf{s}(\omega_n)\rangle=0$. Expanding $\omega_n$ by definition we get:
$$\langle \lambda, \sum_{i=0}^{n-1} \sum_{(\alpha,\beta)} (-1)^i\,
\tau\cdot \left(
\begin{smallmatrix}
1 & \dots & i & i+1 & \dots & n \\
\beta(i) & \dots & \beta(1) & \alpha(1) & \dots & \alpha(n-i)
\end{smallmatrix} \right)^{-1}
\rangle=0$$
Then it is equivalent to the following statement 
$$ \sum_{i=0}^{n-1} \sum_{(\alpha,\beta)} (-1)^i \lambda\left( \tau\cdot \left(
\begin{smallmatrix}
1 & \dots & i & i+1 & \dots & n \\
\beta(i) & \dots & \beta(1) & \alpha(1) & \dots & \alpha(n-i)
\end{smallmatrix}\right)^{-1}\right)=0.$$
\end{proof}

\section{Jacobi subsets}
\begin{Definition} The subset $T\subseteq S_n$ is called a \textit{Jacobi subset} if the 
following identity holds in any Lie ring: $$\sum_{\sigma\in S_n} [x_{\sigma(1)}, \dots , x_{\sigma(n)}] = 0.$$
\end{Definition}

\noindent So, we can consider Jacobi subset $T$ as a Jacobi element $a$ in $\mathbb{Z}[S_n]$ by defining it as 
$a=\sum_{\sigma\in S_n}\alpha_T(\sigma)  \sigma$, where coefficient is defined as $\alpha_T(\sigma)=\{\begin{smallmatrix} 1,\, \sigma\in T \\ 0,\, \sigma\not\in T 
\end{smallmatrix}$.

\noindent Sets $\mathrm{I}_n^+$ and $\mathrm{I}_n^-$ are defined as follows:
$$\mathrm{I}_n^+ =\left\{\left(
\begin{smallmatrix} 
1 & \dots & i & i+1 & \dots & n \\
\beta(i) & \dots & \beta(1) & \alpha(1) & \dots & \alpha(n-i)
\end{smallmatrix}\right)^{-1}\mid (\alpha,\beta) \in Sh^1(n-i, i), \,i \mbox{ is even}, \ 0\leq i\leq n-1 \right\},$$
$$\mathrm{I}_n^- =\left\{\left(
\begin{smallmatrix} 
1 & \dots & i & i+1 & \dots & n \\
\beta(i) & \dots & \beta(1) & \alpha(1) & \dots & \alpha(n-i)
\end{smallmatrix}\right)^{-1}\mid (\alpha,\beta) \in Sh^1(n-i, i), \,  i \mbox{ is odd}, \ 0\leq i\leq n-1 \right\}.$$

\begin{Theorem}
Let $T$ be a subset of $S_n$.  Then $T$ is Jacobi if and only if for any $\tau\in S_n$
$$  \mid T \cap \tau \mathrm{I}_n^+ \mid = 
\mid T \cap \tau \mathrm{I}_n^- \mid.$$
\end{Theorem}

\begin{proof}
As it was said in the beginning any Jacobi subset $T$ can be represented as a Jacobi element $a=\sum_{\sigma\in S_n}\alpha_T(\sigma)  \sigma$ and theorem \ref{th1}. can be applied to it. So, the subset $T$ is Jacobi if and only if  
$$\forall \tau \in S_n \quad \sum_{i=0}^{n-1} \sum_{(\alpha,\beta)} (-1)^i \alpha_T\left( \tau\cdot \left(
\begin{smallmatrix}
1 & \dots & i & i+1 & \dots & n \\
\beta(i) & \dots & \beta(1) & \alpha(1) & \dots & \alpha(n-i)
\end{smallmatrix}\right)^{-1}\right)=0,$$
where $(\alpha,\beta)\in {\sf Sh}^{1}(n-i,i).$  We can move terms with odd $i$ to the right side.
$$\sum_{2\, | \, i} \sum_{(\alpha,\beta)} \alpha_T\left( \tau\cdot \left(
\begin{smallmatrix}
1 & \dots & i & i+1 & \dots & n \\
\beta(i) & \dots & \beta(1) & \alpha(1) & \dots & \alpha(n-i)
\end{smallmatrix}\right)^{-1}\right)=
\sum_{2\, \nmid \, i}\sum_{(\alpha,\beta)} \alpha_T\left( \tau\cdot \left(
\begin{smallmatrix}
1 & \dots & i & i+1 & \dots & n \\
\beta(i) & \dots & \beta(1) & \alpha(1) & \dots & \alpha(n-i)
\end{smallmatrix}\right)^{-1}\right)$$
Notice that coefficient $\alpha_T$ shows that some permutation belongs to the subset $T$ or not. Hence sum of these coefficients is equal to cardinality of the intersection of $T$ and the set of considered permutations. Then the following is satisfied: 
$$\forall \tau \in S_n \quad \sum_{2\, | \,i} \sum_{(\alpha,\beta)} \alpha_T\left( \tau\cdot \left(
\begin{smallmatrix}
1 & \dots & i & i+1 & \dots & n \\
\beta(i) & \dots & \beta(1) & \alpha(1) & \dots & \alpha(n-i)
\end{smallmatrix}\right)^{-1}\right)=\mid T \cap \tau \mathrm{I}_n^+ \mid,$$
$$\forall \tau \in S_n \quad \sum_{2\, \nmid \, i}\sum_{(\alpha,\beta)} \alpha_T\left( \tau\cdot \left(
\begin{smallmatrix}
1 & \dots & i & i+1 & \dots & n \\
\beta(i) & \dots & \beta(1) & \alpha(1) & \dots & \alpha(n-i)
\end{smallmatrix}\right)^{-1}\right) = \mid T \cap \tau \mathrm{I}_n^-\mid.$$
So, the previous equality can be represented as 
$$\mid T \cap \tau \mathrm{I}_n^+ \mid = \mid T \cap \tau \mathrm{I}_n^- \mid,\quad  \forall \tau \in S_n.$$ 
As a consequence $T$ is Jacobi if and only if $$\quad \mid T \cap \tau \mathrm{I}_n^+ \mid = 
\mid T \cap \tau \mathrm{I}_n^- \mid, \quad \forall \tau \in S_n.$$
\end{proof}

\end{document}